\documentclass[12pt,reqno]{amsart}
\usepackage{eurosym}
\usepackage{amsmath,amsthm,amscd,amsfonts,amssymb,graphicx,color}
\usepackage[bookmarksnumbered,colorlinks,plainpages]{hyperref}

\hypersetup{colorlinks=true,linkcolor=red, anchorcolor=green,
citecolor=cyan, urlcolor=red, filecolor=magenta, pdftoolbar=true}
\newtheorem{theorem}{Theorem}
\newtheorem{lemma}{Lemma}

\theoremstyle{definition}
\newtheorem{definition}{Definition}

\theoremstyle{remark}
\newtheorem{remark}{Remark}
\numberwithin{equation}{section}

\begin{document}

\setcounter{page}{1}

\begin{center}
{\Large \textbf{Approximation by $q$-Bernstein-Stancu-Kantorovich operators with
shifted knots of real parameters}}

\bigskip

\textbf{Mohammad Ayman Mursaleen}$^{1,2,\ast}$, \textbf{Adem Kilicman}$^{2}$ and \textbf{Md. Nasiruzzaman}$^{3}$\\
\bigskip
$^{1}$School of Information and Physical Sciences, \\The University of Newcastle, Callaghan, NSW 2308, Australia\\[0pt]
$^{2}$Department of Mathematics and Statistics, Faculty of Science, \\University Putra Malaysia, 43400 Serdang, Selangor, Malaysia\\[0pt]
$^{3}$Department of Mathematics, Faculty of Science, \\University of Tabuk,
PO Box 4279, Tabuk-71491, Saudi Arabia\\[0pt]
\bigskip
mohdaymanm@gmail.com; akilic@upm.edu.my; nasir3489@gmail.com
\end{center}
\bigskip
\textbf{Abstract}
Our main purpose of this article is to study the convergence and other
related properties of $q$-Bernstein-Kantorovich operators including the
shifted knots of real positive numbers. We design the shifted
knots of Bernstein-Kantorovich operators generated by the basic $q$%
-calculus. More precisely, we study the convergence properties of our new
operators in the space of continuous functions and Lebesgue space. We obtain
the degree of convergence with the help of modulus of continuity and
integral modulus of continuity. Furthermore, we establish the quantitative
estimates of Voronovskaja-type.

\bigskip



\bigskip

\section{\textbf{Introduction and Basic Definitions }}

In 1912, S. N. Bernstein gave a very short proof of Weierstrass
approximation theorem by introducing the Bernstein polynomials \cite{bns}.
The $q$-analog of Bernstein polynomials were introduced by Lupa\c{s} \cite%
{lups} and Phillips \cite{phps}, separately. Since then several polynomials
were generalized and studied by using $q$-calculus. For example related to
our present theme, $q$-Bernstein shifted operators \cite{kjn}, $q$%
-Bernstein-Kantorovich operators \cite{mfa1}, Bernstein-Kantorovich
operators based on $(p,q)$-calculus \cite{maa}, and other related operators
\cite{maths} etc..

In 2004, the Stancu variant of Bernstein-Kantorovich operators \cite{barbosu}
were defined as follows:
\begin{equation}
\mathcal{J}_{r}^{\mu _{1},\nu _{1}}(g;\xi )=(r+1+\nu _{1})\sum_{s=0}^{r}%
\binom{m}{s}\xi ^{s}(1-\xi )^{r-s}\int_{\frac{s+\mu _{1}}{r+1+\nu _{1}}}^{%
\frac{s+1+\mu _{1}}{r+1+\nu _{1}}}f(t)\mathrm{d}t,\text{ }(0\leq \mu
_{1}\leq \nu _{1}).  \label{mnz-1}
\end{equation}

Mohiuddine \emph{et al} constructed a new family of Bernstein- Kantorovich
operators \cite{samtaaammas}. While the operators \eqref{mnz-1} with shifted
knots were studied by Gadjiev et al. \cite{gad} by defining the operators $%
\mathcal{S}_{r,\mu _{2}}^{\nu _{2}}:C[0,1]\rightarrow C\bigg{[}\frac{\mu _{2}%
}{r+\nu _{2}},\frac{r+\mu _{2}}{r+\nu _{2}}\bigg{]}$ as follows:
\begin{equation*}
\mathcal{S}_{r,\mu _{2}}^{\nu _{2}}(f;\xi )=\left( \frac{r+\nu _{2}}{r}%
\right) ^{r}\sum_{s=0}^{r}\binom{r}{s}\left( \xi -\frac{\mu _{2}}{r+\nu _{2}}%
\right) _{q}^{s}\left( \frac{r+\mu _{2}}{r+\nu _{2}}-\xi \right)
_{q}^{r-s}f\left( \frac{t+\mu }{r+\nu }\right) ,
\end{equation*}%
where $\frac{\mu _{2}}{r+\nu _{2}}\leq z\leq \frac{r+\mu _{2}}{r+\nu _{2}}$
and $\mu _{k},\nu _{k}(k=1,2)$ are positive real numbers satisfying $0\leq
\mu _{2}\leq \mu _{1}\leq \nu _{1}\leq \nu _{2}$. Moreover, Sucu et al. \cite%
{nfs} introduced the Bernstein-Stancu-Kantorovich operators\ with shifted
knots by:
\begin{eqnarray*}
\mathcal{K}_{r,\mu _{1},\nu _{1}}^{\mu _{2},\nu _{2}}(f;\xi ) &=&\left(
\frac{r+\nu _{2}+1}{r+1}\right) ^{r+1}\sum_{s=0}^{r}\left( \int_{\frac{r+\mu
_{1}}{r+1+\nu _{1}}}^{\frac{r+\mu _{1}+1}{r+1+\nu _{1}}}f(\xi )\mathrm{d}\xi
\right) \\
&\times &\binom{r}{s}\left( \xi -\frac{\mu _{2}}{r+\nu _{2}+1}\right)
^{s}\left( \frac{r+\mu _{2}+1}{r+\nu _{2}+1}-\xi \right) ^{r-s}.
\end{eqnarray*}

Here we introduce the Bernstein-Stancu-Kantorovich operators with shifted
knots by using $q$-integers and then investigate the approximation in the $%
L_{p}\;(p\geq 1)$ spaces. We obtain the degree of approximation of our new
operators by using the modulus of continuity and integral modulus of
continuity. We give some direct theorems in $L_{p}$ spaces. Furthermore, we
obtain the approximation in Lipschitz spaces and also establishe the
quantitative estimates of the Voronovskaja-type theorem. The $q$-calculus
emerged as a very useful tool and a very fruitful connection between
Mathematics and Physics. We recall here some basic definitions and facts
about the $q$-calculus. We take $\mathbb{C}$ as the set of complex numbers
and $\mathbb{N}$ the set of positive integers.\newline

\begin{definition}
\label{def1} \textrm{Let $q\in \mathbb{C}\setminus \{0,1\}.$ Then the $q$%
-number is defined by}

\begin{equation}  \label{q-number}
\left[\theta\right]_{q}=\left\{
\begin{array}{ll}
\dfrac{1-q^{\theta}}{1-q} & \qquad (\theta\in \mathbb{C}\setminus\{0\}) \\
&  \\
1 & \qquad (\theta=0) \\
&  \\
\sum\limits_{s=0}^{r-1}q^{s}=1+q+q^2+\cdots+q^{s} & \qquad (\theta\in
\mathbb{N}).%
\end{array}
\right.
\end{equation}
\end{definition}

\begin{definition}
\label{def2} \textrm{For number $q\in \mathbb{C}\setminus \{0,1\},$ the $q$%
-factorial is defined by}

\begin{equation}  \label{q-factorial}
\left[ \theta\right] _{q}!=\left\{
\begin{array}{ll}
1 & \qquad (\theta=0) \\
&  \\
\prod\limits_{\theta=1}^{r}\left[\theta\right]_{q} & \qquad (\theta\in
\mathbb{N}).%
\end{array}
\right.
\end{equation}
\end{definition}

\begin{definition}
\label{def2a} \textrm{For $q\in \mathbb{C}\setminus \{0,1\},$ and $0\leq
s\leq r$ the $q$-binomial coefficient is defined by}
\begin{equation*}
{\Huge [}_{{\Huge s}}^{{\Huge r}}{\Huge ]}{=}\frac{[r]_{q}!}{%
[s]_{q}!\;[r-s]_{q}!}.\text{ }
\end{equation*}
\end{definition}

The $q$-analogue of $(1+\xi )^{r}$ is the polynomial given by
\begin{equation}
(1+\xi )_{q}^{r-s}=\left\{
\begin{array}{ll}
(1+\xi )(1+q\xi )\cdots (1+q^{r-s-1}\xi ) & \qquad (r,s\in \mathbb{N}) \\
&  \\
1 & \qquad (r=s=0).%
\end{array}%
\right.
\end{equation}

\begin{definition}
\label{defin2} \cite{jkson,kak} The $q$-Jackson integral from $0$ to $A\in
\mathbb{R}$ is defined by

\begin{equation}
\int_{0}^{A}f(\xi )\mathrm{d}_{q}\xi =A(1-q)\sum_{s=0}^{\infty
}f(Aq^{s})q^{s},  \label{q-intial-2}
\end{equation}%
while, the more general $q$-Jackson integral on interval $[A,B]$ is given by

\begin{equation}
\int_{A}^Bf(z)\mathrm{d}_qz=\int_{0}^Bf(z)\mathrm{d}_qz- \int_{0}^Af(z)%
\mathrm{d}_qz.  \label{q-intial-3}
\end{equation}
\end{definition}

\begin{theorem}
\cite{kak} \textbf{(Fundamental theorem of $q$-calculus )} Let $\varphi $ be
the anti $q$-derivative of the function $f$, namely $D_{q}\varphi =f$, be
continuous at $z=A$. Then
\begin{equation}
\int_{A}^{B}f(\xi )\mathrm{d}_{q}\xi =\varphi (B)-\varphi (A),
\label{q-funamenal}
\end{equation}

\begin{equation}
D_{q}\left( \int_{A}^{\xi }f(t)\mathrm{d}_{q}t\right) =f(\xi ).
\label{q-funamenal-2}
\end{equation}
\end{theorem}

\section{\textbf{Construction of operators and basic properties}}

In this section, taking into account the operators introduced by \cite%
{barbosu,kant}, we focused on these operators and then construct the new
generalized operators in $L_{p}$ spaces by $q$-analogue and investigate the
convergence results. Thus, with the basic definitions of $q$-integers for $%
0<q<1$, we define new construction of Bernstein-Kantorovich polynomials with
shifted knots of positive real numbers $\mu _{s},\nu _{s}$ for $(s=1,2)$
with $\frac{\mu _{2}}{[r+1]_{q}+\nu _{2}}\leq \xi \leq \frac{\lbrack
r+1]_{q}+\mu _{2}}{[r+1]_{q}+\nu _{2}},$ provided $0\leq \mu _{2}\leq \mu
_{1}\leq \nu _{1}\leq \nu _{2}$. For this purpose we let $1\leq p<\infty ,\;%
\mathcal{J}_{r}=\left[ \frac{\mu _{2}}{[r+1]_{q}+\nu _{2}},\frac{%
[r+1]_{q}+\mu _{2}}{[r+1]_{q}+\nu _{2}}\right] $ and define the operators $%
\mathcal{B}_{r,q,\mu _{1},\nu _{1}}^{\mu _{2},\nu
_{2}}:L_{p}[0,1]\rightarrow L_{p}(\mathcal{J}_{r})$. Then for all $f\in
L_{p}[0,1]$ and $r=1,2,3,\cdots $ we construct the sequence of operators by
\begin{equation}
\mathcal{B}_{r,q,\mu _{1},\nu _{1}}^{\mu _{2},\nu _{2}}(f;\xi
)=([r+1]_{q}+\nu _{1})\left( \frac{[r+1]_{q}+\nu _{2}}{[r+1]_{q}}\right)
_{q}^{r+1}\sum_{s=0}^{r}P_{r,q,\mu _{1},\nu _{1}}^{\mu _{2},\nu _{2}}(\xi
)\int_{\frac{q[s]_{q}+\mu _{1}}{[r+1]_{q}+\nu _{1}}}^{\frac{[s+1]_{q}+\mu
_{1}}{[r+1]_{q}+\nu _{1}}}f(t)\mathrm{d}t,  \label{d-1}
\end{equation}%
where
\begin{equation*}
P_{r,q,\mu _{1},\nu _{1}}^{\mu _{2},\nu _{2}}(\xi )=\left[
\begin{array}{c}
r \\
s%
\end{array}%
\right] _{q}\left( \xi -\frac{\mu _{2}}{[r+1]_{q}+\nu _{2}}\right)
_{q}^{s}\prod_{i=0}^{r-s-1}\left( \frac{[r+1]_{q}+\mu _{2}}{[r+1]_{q}+\nu
_{2}}-q^{i}\xi \right)
\end{equation*}%
and
\begin{equation*}
\prod_{i=0}^{r-s-1}\left( \frac{[r+1]_{q}+\mu _{2}}{[r+1]_{q}+\nu _{2}}%
-q^{i}\xi \right) =\left( \frac{[r+1]_{q}+\mu _{2}}{[r+1]_{q}+\nu _{2}}-\xi
\right) _{q}^{r-s}.
\end{equation*}%
In addition, it is very clear that for $q=1,$ operators reduce to \cite{nfs}
and for $\mu _{2}=\nu _{2}=0$ with $q=1$ we get operators \eqref{mnz-1} by
\cite{barbosu} . If we take $\mu _{1}=\mu _{2}=\nu _{1}=\nu _{2}=0$ with $%
q=1,$ then the operators \eqref{mnz-1} reduce to the classic
Bernstein-Kantorovich operators \cite{kant}.\newline

\begin{lemma}
\label{sn} Take the test functions $\eta _{j}(t)=t^{j}$ for $j=0,1,2.$ Then
for all $\xi \in \mathcal{J}_{r}=\left[ \frac{\mu _{2}}{[r+1]_{q}+\nu _{2}},%
\frac{[r+1]_{q}+\mu _{2}}{[r+1]_{q}+\nu _{2}}\right] ,$ we have:
\begin{eqnarray*}
(1)\quad \mathcal{B}_{r,q,\mu _{1},\nu _{1}}^{\mu _{2},\nu _{2}}(\eta
_{0}(t);\xi ) &=&\left( \frac{[r+1]_{q}+\nu _{2}}{[r+1]_{q}}\right) _{q}; \\
(2)\quad \mathcal{B}_{r,q,\mu _{1},\nu _{1}}^{\mu _{2},\nu _{2}}(\eta
_{1}(t);\xi ) &=&\frac{2q}{[2]_{q}}\frac{[r]_{q}}{[r+1]_{q}+\nu _{1}}\left(
\frac{[r+1]_{q}+\nu _{2}}{[r+1]_{q}}\right) _{q}^{2}\left( \xi -\frac{\mu
_{2}}{[r+1]_{q}+\nu _{2}}\right) _{q} \\
&+&\frac{(1+2\mu _{1})}{[2]_{q}([r+1]_{q}+\nu _{1})}\left( \frac{%
[r+1]_{q}+\nu _{2}}{[r+1]_{q}}\right) _{q}; \\
(3)\quad \mathcal{B}_{r,q,\mu _{1},\nu _{1}}^{\mu _{2},\nu _{2}}(\eta
_{2}(t);\xi ) &=&\frac{3q^{3}}{[3]_{q}}\frac{[r]_{q}[r-1]_{q}}{%
([r+1]_{q}+\nu _{1})^{2}}\left( \frac{[r+1]_{q}+\nu _{2}}{[r+1]_{q}}\right)
_{q}^{3}\left( \xi -\frac{\mu _{2}}{[r+1]_{q}+\nu _{2}}\right) _{q}^{2} \\
&+&\frac{3q(q+1+2\mu _{1})}{[3]_{q}}\frac{[r]_{q}}{([r+1]_{q}+\nu _{1})^{2}}%
\left( \frac{[r+1]_{q}+\nu _{2}}{[r+1]_{q}}\right) _{q}^{2}\left( \xi -\frac{%
\mu _{2}}{[r+1]_{q}+\nu _{2}}\right) _{q} \\
&+&\frac{1}{[3]_{q}}\frac{(1+3\mu _{1}+3\mu _{1}^{2})}{([r+1]_{q}+\nu
_{1})^{2}}\left( \frac{[r+1]_{q}+\nu _{2}}{[r+1]_{q}}\right) _{q}.
\end{eqnarray*}
\end{lemma}

\begin{proof}
From $q$-Jackson integral, clearly we have
\begin{eqnarray*}
\int_{\frac{q[s]_{q}+\mu _{1}}{[r+1]_{q}+\nu _{1}}}^{\frac{[s+1]_{q}+\mu _{1}%
}{[r+1]_{q}+\nu _{1}}}t^{\alpha }\mathrm{d}_{q}t &=&\int_{0}^{\frac{%
[s+1]_{q}+\mu _{1}}{[r+1]_{q}+\nu _{1}}}t^{\alpha }\mathrm{d}_{q}t-\int_{0}^{%
\frac{q[s]_{q}+\mu _{1}}{[r+1]_{q}+\nu _{1}}}t^{\alpha }\mathrm{d}_{q}t \\
&=&(1-q)\frac{[s+1]_{q}+\mu _{1}}{[r+1]_{q}+\nu _{1}}\sum_{m=0}^{\infty }%
\bigg{(}\frac{[s+1]_{q}+\mu _{1}}{[r+1]_{q}+\nu _{1}}q^{m}\bigg{)}^{\alpha
}q^{m} \\
&-&(1-q)\frac{q[s]_{q}+\mu _{1}}{[r+1]_{q}+\nu _{1}}\sum_{m=0}^{\infty }%
\bigg{(}\frac{q[s]_{q}+\mu _{1}}{[r+1]_{q}+\nu _{1}}q^{m}\bigg{)}^{\alpha
}q^{m} \\
&=&\frac{(1-q)}{([r+1]_{q}+\nu _{1})^{\alpha +1}}\bigg{(}([s+1]_{q}+\mu
_{1})^{\alpha +1}-(q[s]_{q}+\mu _{1})^{\alpha +1}\bigg{)}\sum_{m=0}^{\infty
}q^{m(1+\alpha )}.
\end{eqnarray*}%
Thus we easily get
\begin{equation}
\int_{\frac{q[s]_{q}+\mu _{1}}{[r+1]_{q}+\nu _{1}}}^{\frac{[s+1]_{q}+\mu _{1}%
}{[r+1]_{q}+\nu _{1}}}t^{\alpha }\mathrm{d}_{q}t=%
\begin{cases}
\frac{1}{[r+1]_{q}+\nu _{1}} & \quad \text{for }\;\alpha =0; \\
\frac{1}{[2]_{q}\left( [r+1]_{q}+\nu _{1}\right) ^{2}}\left( 1+2\mu
_{1}+2q[r]_{q}\right) & \quad \text{for }\;\alpha =1; \\
\frac{1}{[3]_{q}\left( [r+1]_{q}+\nu _{1}\right) ^{3}}\bigg{(}1+3\mu
_{1}+3\mu _{1}^{2}+3q(1+2\mu _{1})[r]_{q}+3q^{2}[r]_{q}^{2}\bigg{)} & \quad
\text{for }\;\alpha =2.%
\end{cases}
\label{kantro}
\end{equation}%
where we used $[s+1]_{q}=1+q[s]_{q}$.\newline

Thus in the view of \eqref{kantro} for $\alpha =0,1,2,\;\eta _{\alpha
}(t)=t^{\alpha },$ we get

\begin{eqnarray*}
\mathcal{B}_{r,q,\mu _{1},\nu _{1}}^{\mu _{2},\nu _{2}}(\eta _{0}(t);\xi )
&=&([r+1]_{q}+\nu _{1})\left( \frac{[r+1]_{q}+\nu _{2}}{[r+1]_{q}}\right)
_{q}^{r+1}\frac{1}{([r+1]_{q}+\nu _{1})}\sum_{s=0}^{r}P_{r,q,\mu _{1},\nu
_{1}}^{\mu _{2},\nu _{2}}(\xi ) \\
&=&\left( \frac{[r+1]_{q}+\nu _{2}}{[r+1]_{q}}\right) _{q}^{r+1}\left( \frac{%
[r+1]_{q}}{[r+1]_{q}+\nu _{2}}\right) _{q}^{r}.
\end{eqnarray*}
\begin{eqnarray*}
\mathcal{B}_{r,q,\mu _{1},\nu _{1}}^{\mu _{2},\nu _{2}}(\eta _{1}(t);\xi)
&=&\frac{(1+2\mu _{1})([r+1]_{q}+\nu _{1})}{[2]_{q}}\left( \frac{%
[r+1]_{q}+\nu _{2}}{[r+1]_{q}}\right) _{q}^{r+1}\frac{1}{([r+1]_{q}+\nu
_{1})^{2}}\sum_{s=0}^{r}P_{r,q,\mu _{1},\nu _{1}}^{\mu _{2},\nu _{2}}(\xi )
\\
&+&\frac{2q([r+1]_{q}+\nu _{1})}{[2]_{q}}\left( \frac{[r+1]_{q}+\nu _{2}}{%
[r+1]_{q}}\right) _{q}^{r+1}\frac{1}{([r+1]_{q}+\nu _{1})^{2}}%
\sum_{s=0}^{r}[s]_{q}P_{r,q,\mu _{1},\nu _{1}}^{\mu _{2},\nu _{2}}(\xi ) \\
&=&\frac{(1+2\mu _{1})}{[2]_{q}([r+1]_{q}+\nu _{1})}\left( \frac{%
[r+1]_{q}+\nu _{2}}{[r+1]_{q}}\right) _{q}^{r+1}\left( \frac{[r+1]_{q}}{%
[r+1]_{q}+\nu _{2}}\right) _{q}^{r} \\
&+&\frac{2q}{[2]_{q}}\frac{[r]_{q}}{([r+1]_{q}+\nu _{1})}\left( \frac{%
[r+1]_{q}+\nu _{2}}{[r+1]_{q}}\right) _{q}^{r+1}\sum_{s=0}^{r}P_{r-1,q,\mu
_{1},\nu _{1}}^{\mu _{2},\nu _{2}}(\xi )\left( \xi -\frac{\mu _{2}}{%
[r+1]_{q}+\nu _{2}}\right) _{q} \\
&=&\frac{(1+2\mu _{1})}{[2]_{q}([r+1]_{q}+\nu _{1})}\left( \frac{%
[r+1]_{q}+\nu _{2}}{[r+1]_{q}}\right) _{q} \\
&+&\frac{2q}{[2]_{q}}\frac{[r]_{q}}{([r+1]_{q}+\nu _{1})}\left( \frac{%
[r+1]_{q}+\nu _{2}}{[r+1]_{q}}\right) _{q}^{r+1}\left( \xi -\frac{\mu _{2}}{%
[r+1]_{q}+\nu _{2}}\right) _{q}\left( \frac{[r+1]_{q}}{[r+1]_{q}+\nu _{2}}%
\right) _{q}^{r-1};
\end{eqnarray*}
$\mathcal{B}_{r,q,\mu _{1},\nu _{1}}^{\mu _{2},\nu _{2}}(\eta _{2}(t);\xi )$
\begin{eqnarray*}
&=&\frac{1}{[3]_{q}}\frac{(1+3\mu _{1}+3\mu _{1}^{2})}{([r+1]_{q}+\nu
_{1})^{2}}\left( \frac{[r+1]_{q}+\nu _{2}}{[r+1]_{q}}\right)
_{q}^{r+1}\sum_{s=0}^{r}P_{r,q,\mu _{1},\nu _{1}}^{\mu _{2},\nu _{2}}(\xi )
\\
&+&\frac{3q(1+2\mu _{1})}{[3]_{q}}\frac{[r]_{q}}{([r+1]_{q}+\nu _{1})^{2}}%
\left( \frac{[r+1]_{q}+\nu _{2}}{[r+1]_{q}}\right)
_{q}^{r+1}\sum_{s=0}^{r}P_{r-1,q,\mu _{1},\nu _{1}}^{\mu _{2},\nu _{2}}(\xi
)\left( \xi -\frac{\mu _{2}}{[r+1]_{q}+\nu _{2}}\right) _{q} \\
&+&\frac{3q^{2}}{[3]_{q}}\frac{[r]_{q}}{([r+1]_{q}+\nu _{1})^{2}}\left(
\frac{[r+1]_{q}+\nu _{2}}{[r+1]_{q}}\right)
_{q}^{r+1}\sum_{s=0}^{r}[1+s]_{q}P_{r-1,q,\mu _{1},\nu _{1}}^{\mu _{2},\nu
_{2}}(\xi )\left( \xi -\frac{\mu _{2}}{[r+1]_{q}+\nu _{2}}\right) _{q} \\
&=&\frac{1}{[3]_{q}}\frac{(1+3\mu _{1}+3\mu _{1}^{2})}{([r+1]_{q}+\nu
_{1})^{2}}\left( \frac{[r+1]_{q}+\nu _{2}}{[r+1]_{q}}\right) _{q} \\
&+&\frac{3q(1+2\mu _{1})}{[3]_{q}}\frac{[r]_{q}}{([r+1]_{q}+\nu _{1})^{2}}%
\left( \frac{[r+1]_{q}+\nu _{2}}{[r+1]_{q}}\right) _{q}^{2}\left( \xi -\frac{%
\mu _{2}}{[r+1]_{q}+\nu _{2}}\right) _{q} \\
&+&\frac{3q^{2}}{[3]_{q}}\frac{[r]_{q}}{([r+1]_{q}+\nu _{1})^{2}}\left(
\frac{[r+1]_{q}+\nu _{2}}{[r+1]_{q}}\right)
_{q}^{r+1}\sum_{s=0}^{r}P_{r-1,q,\mu _{1},\nu _{1}}^{\mu _{2},\nu _{2}}(\xi
)\left( \xi -\frac{\mu _{2}}{[r+1]_{q}+\nu _{2}}\right) _{q} \\
&+&\frac{3q^{3}}{[3]_{q}}\frac{[r]_{q}[r-1]_{q}}{([r+1]_{q}+\nu _{1})^{2}}%
\left( \frac{[r+1]_{q}+\nu _{2}}{[r+1]_{q}}\right)
_{q}^{r+1}\sum_{s=0}^{r}P_{r-2,q,\mu _{1},\nu _{1}}^{\mu _{2},\nu _{2}}(\xi
)\left( \xi -\frac{\mu _{2}}{[r+1]_{q}+\nu _{2}}\right) _{q}^{2} \\
&=&\frac{1}{[3]_{q}}\frac{(1+3\mu _{1}+3\mu _{1}^{2})}{([r+1]_{q}+\nu
_{1})^{2}}\left( \frac{[r+1]_{q}+\nu _{2}}{[r+1]_{q}}\right) _{q} \\
&+&\frac{3q(1+2\mu _{1})}{[3]_{q}}\frac{[r]_{q}}{([r+1]_{q}+\nu _{1})^{2}}%
\left( \frac{[r+1]_{q}+\nu _{2}}{[r+1]_{q}}\right) _{q}^{2}\left( \xi -\frac{%
\mu _{2}}{[r+1]_{q}+\nu _{2}}\right) _{q} \\
&+&\frac{3q^{2}}{[3]_{q}}\frac{[r]_{q}}{([r+1]_{q}+\nu _{1})^{2}}\left(
\frac{[r+1]_{q}+\nu _{2}}{[r+1]_{q}}\right) _{q}^{2}\left( \xi -\frac{\mu
_{2}}{[r+1]_{q}+\nu _{2}}\right) _{q} \\
&+&\frac{3q^{3}}{[3]_{q}}\frac{[r]_{q}[r-1]_{q}}{([r+1]_{q}+\nu _{1})^{2}}%
\left( \frac{[r+1]_{q}+\nu _{2}}{[r+1]_{q}}\right) _{q}^{3}\left( \xi -\frac{%
\mu _{2}}{[r+1]_{q}+\nu _{2}}\right) _{q}^{2}.
\end{eqnarray*}

Which completes the proof of Lemma \ref{sn}.
\end{proof}

\begin{lemma}
\label{snlm1}For the operators $\mathcal{B}_{r,q,\mu _{1},\nu _{1}}^{\mu
_{2},\nu _{2}}(.~;~.)$ defined by \eqref{d-1}, we have the following central
moments:\newline
\newline
$(1)\quad \mathcal{B}_{r,q,\mu _{1},\nu _{1}}^{\mu _{2},\nu _{2}}(\eta
_{1}(t)-\xi ;\xi )$
\begin{eqnarray*}
&=&\bigg{[}\frac{2q}{[2]_{q}}\frac{[r]_{q}}{[r+1]_{q}+\nu _{1}}\left( \frac{%
[r+1]_{q}+\nu _{2}}{[r+1]_{q}}\right) _{q}^{2}-\left( \frac{[r+1]_{q}+\nu
_{2}}{[r+1]_{q}}\right) _{q}\bigg{]}\xi \\
&+&\bigg{[}\frac{1+2\mu _{1}}{[2]_{q}([r+1]_{q}+\nu _{1})}\left( \frac{%
[r+1]_{q}+\nu _{2}}{[r+1]_{q}}\right) _{q}-\frac{2q\mu _{2}}{[2]_{q}}\frac{%
[r]_{q}}{([r+1]_{q}+\nu _{1})([r+1]_{q}+\nu _{2})}\left( \frac{[r+1]_{q}+\nu
_{2}}{[r+1]_{q}}\right) _{q}^{2}\bigg{]};
\end{eqnarray*}%
$(2)\quad \mathcal{B}_{r,q,\mu _{1},\nu _{1}}^{\mu _{2},\nu _{2}}((\eta
_{2}(t)-\xi )^{2};\xi )$
\begin{eqnarray*}
&=&\bigg{[}\frac{3q^{4}}{[3]_{q}}\frac{[r]_{q}[r-1]_{q}}{([r+1]_{q}+\nu
_{1})^{2}}\left( \frac{[r+1]_{q}+\nu _{2}}{[r+1]_{q}}\right) _{q}^{3} \\
&-&\frac{4q}{[2]_{q}}\frac{[r]_{q}}{([r+1]_{q}+\nu _{1})}\left( \frac{%
[r+1]_{q}+\nu _{2}}{[r+1]_{q}}\right) _{q}^{2}+\left( \frac{[r+1]_{q}+\nu
_{2}}{[r+1]_{q}}\right) _{q}\bigg{]}\xi ^{2} \\
&+&\bigg{[}\frac{3q(q+1+2\mu _{1})}{[3]_{q}}\frac{[r]_{q}}{([r+1]_{q}+\nu
_{1})^{2}}\left( \frac{[r+1]_{q}+\nu _{2}}{[r+1]_{q}}\right) _{q}^{2} \\
&+&\frac{4q\mu _{2}}{[2]_{q}}\frac{[r]_{q}}{([r+1]_{q}+\nu
_{1})([r+1]_{q}+\nu _{2})}\left( \frac{[r+1]_{q}+\nu _{2}}{[r+1]_{q}}\right)
_{q}^{2} \\
&-&\frac{3q^{3}(1+q)\mu _{2}}{[3]_{q}([r+1]_{q}+\nu _{2})}\frac{%
[r]_{q}[r-1]_{q}}{([r+1]_{q}+\nu _{1})^{2}}\left( \frac{[r+1]_{q}+\nu _{2}}{%
[r+1]_{q}}\right) _{q}^{3} \\
&-&\frac{2(1+2\mu _{1})}{[2]_{q}([r+1]_{q}+\nu _{1})}\left( \frac{%
[r+1]_{q}+\nu _{2}}{[r+1]_{q}}\right) _{q}\bigg{]}\xi \\
&+&\frac{3q^{3}}{[3]_{q}}\frac{[r]_{q}[r-1]_{q}}{([r+1]_{q}+\nu _{1})^{2}}%
\frac{\mu _{2}}{([r+1]_{q}+\nu _{2})^{2}}\left( \frac{[r+1]_{q}+\nu _{2}}{%
[r+1]_{q}}\right) _{q}^{2} \\
&-&\frac{3q\mu _{2}(1+q+2\mu _{1})}{[3]_{q}([r+1]_{q}+\nu _{2})}\frac{[r]_{q}%
}{([r+1]_{q}+\nu _{1})^{2}}\left( \frac{[r+1]_{q}+\nu _{2}}{[r+1]_{q}}%
\right) _{q}^{2} \\
&+&\frac{1+3\mu _{1}+3\mu _{1}^{2}}{[3]_{q}([r+1]_{q}+\nu _{1})^{2}}\left(
\frac{[r+1]_{q}+\nu _{2}}{[r+1]_{q}}\right) _{q}.
\end{eqnarray*}
\end{lemma}

\begin{theorem}
\label{tm-1} Let $z\in \mathcal{J}_{r}.$ Then for every $f\in C(\mathcal{J}%
_{r}),$ we get
\begin{equation*}
\lim_{r\rightarrow \infty }||\mathcal{B}_{r,q,\mu _{1},\nu _{1}}^{\mu
_{2},\nu _{2}}(f;\xi )-f(\xi )||_{C(\mathcal{J}_{r})}=0.
\end{equation*}%
where $\mathcal{J}_{r}=\left[ \frac{\mu _{2}}{[r+1]_{q}+\nu _{2}},\frac{%
[r+1]_{q}+\mu _{2}}{[r+1]_{q}+\nu _{2}}\right] ,$ and $C(\mathcal{J}_{r})$
be the set of all continuous function $f$ on $\mathcal{J}_{r}$.
\end{theorem}

\begin{proof}
In the view of Lemma \ref{sn} for $\kappa =0,1,2,$ we easily get

\begin{equation}
\lim_{r\rightarrow \infty }\max_{\xi \in \mathcal{J}_{r}}|\mathcal{B}%
_{r,q,\mu _{1},\nu _{1}}^{\mu _{2},\nu _{2}}(\eta _{\kappa }(t);\xi )-\xi
^{\kappa }|=0.  \label{3.1}
\end{equation}%
We construct operators $\mathcal{C}_{r,q,\mu _{1},\nu }^{\mu _{2},\nu _{2}}$
as below:
\begin{equation}
\mathcal{C}_{r,q,\mu _{1},\nu _{1}}^{\mu _{2},\nu _{2}}(f;\xi )=%
\begin{cases}
\mathcal{B}_{r,q,\mu _{1},\nu _{1}}^{\mu _{2},\nu _{2}}(f;\xi ) & \quad
\text{if }\;\xi \in \mathcal{J}_{r}, \\
f(\xi ) & \quad \text{if }\;\xi \in \lbrack 0,1]\setminus \mathcal{J}_{r}.%
\end{cases}
\label{tt}
\end{equation}%
Then, it is easy to see that
\begin{equation}
||\mathcal{C}_{r,q,\mu _{1},\nu _{1}}^{\mu _{2},\nu _{2}}(f;\xi )-f(\xi
)||_{C[0,1]}=\max_{\xi \in \mathcal{J}_{r}}|\mathcal{B}_{r,q,\mu _{1},\nu
_{1}}^{\mu _{2},\nu _{2}}(f;\xi )-f(\xi )|.  \label{3.3}
\end{equation}

It is obvious to see from \eqref{3.1} and \eqref{3.3} that
\begin{equation*}
\lim_{r\rightarrow \infty }||\mathcal{C}_{r,q,\mu _{1},\nu _{1}}^{\mu
_{2},\nu _{2}}(\eta _{\kappa }(t);\xi )-\xi ^{\kappa }||_{C[0,1]}=0,\;\kappa
=0,1,2.
\end{equation*}%
On applying the well-known Korovkin's theorem, we get
\begin{equation*}
\lim_{r\rightarrow \infty }||\mathcal{C}_{r,q,\mu _{1},\nu _{1}}^{\mu
_{2},\nu _{2}}(f;\xi )-f(\xi )||_{C[0,1]}=0.
\end{equation*}%
Therefore, from \eqref{3.3} we get
\begin{equation*}
\lim_{r\rightarrow \infty }\max_{\xi \in \mathcal{J}_{r}}|\mathcal{C}%
_{r,q,\mu _{1},\nu _{1}}^{\mu _{2},\nu _{2}}(f;\xi )-f(\xi )|=0.
\end{equation*}%
This completes proof of Theorem \ref{tm-1}.
\end{proof}

\begin{theorem}
\label{sm-1} For every $f\in L_{p}[0,1],\;p\geq 1,$ operators \eqref{d-1}
satisfy
\begin{equation*}
\lim_{r\rightarrow \infty }||\mathcal{B}_{r,q,\mu _{1},\nu _{1}}^{\mu
_{2},\nu _{2}}(f;\xi )-f(\xi )||_{L_{P}[0,1]}=0.
\end{equation*}
\end{theorem}

\begin{proof}
In order to prove the result, we consider Theorem \ref{tm-1} and operators $%
\mathcal{C}_{r,q,\mu _{1},\nu _{1}}^{\mu _{2},\nu _{2}}$ by \eqref{tt}. By
Luzin theorem, for a given $\epsilon >0,$ a continuous function $\psi $ on $%
[0,1]$ exits and satisfies $||f-\psi ||_{L_{p}[0,1]}<\epsilon $. Then for
all $r\in \mathbb{N},$ it is enough to prove that there exits a number $R>0$
such that $||\mathcal{C}_{r,q,\mu _{1},\nu _{1}}^{\mu _{2},\nu
_{2}}||_{L_{P}[0,1]}\leq R$. For this purpose, taking into account Theorem %
\ref{tm-1}, for given $\epsilon >0,$ there exits a positive number $r_{0}\in
\mathbb{N}$ with $r\geq r_{0}$ satisfying $||\mathcal{C}_{r,q,\mu _{1},\nu
_{1}}^{\mu _{2},\nu _{2}}(\psi ;\xi )-\psi (\xi )||_{L_{P}[0,1]}<\epsilon $.
Consider the inequality

\begin{eqnarray}
||\mathcal{C}_{r,q,\mu _{1},\nu _{1}}^{\mu _{2},\nu _{2}}(f;\xi )-f(\xi
)||_{L_{p}[0,1]} &\leq &||\mathcal{C}_{r,q,\mu _{1},\nu _{1}}^{\mu _{2},\nu
_{2}}(\psi ;\xi )-\psi (\xi )||_{C[0,1]}+||f-\psi ||_{L_{p}[0,1]}  \notag
\label{mnn-1} \\
&+&||\mathcal{C}_{r,q,\mu _{1},\nu _{1}}^{\mu _{2},\nu _{2}}(f;\xi )-%
\mathcal{C}_{r,q,\mu _{1},\nu _{1}}^{\mu _{2},\nu _{2}}(\psi ;\xi
)||_{L_{p}[0,1]}.
\end{eqnarray}

Apply the Jensen's inequality to operators \eqref{d-1}, we immediately get%
\newline
$|\mathcal{B}_{r,q,\mu _{1},\nu _{1}}^{\mu _{2},\nu _{2}}(f;\xi )|^{p}$
\begin{eqnarray*}
&\leq &\bigg{\{}([r+1]_{q}+\nu _{1})\left( \frac{[r+1]_{q}+\nu _{2}}{%
[r+1]_{q}}\right) _{q}^{r+1}\sum_{s=0}^{r}P_{r,q,\mu _{1},\nu _{1}}^{\mu
_{2},\nu _{2}}(\xi )\int_{\frac{q[s]_{q}+\mu _{1}}{[r+1]_{q}+\nu _{1}}}^{%
\frac{[s+1]_{q}+\mu _{1}}{[r+1]_{q}+\nu _{1}}}|f(t)|\mathrm{d}t\bigg{\}}^{p}
\\
&\leq &\sum_{s=0}^{r}\left( \frac{[r+1]_{q}+\nu _{2}}{[r+1]_{q}}\right)
_{q}^{r}P_{r,q,\mu _{1},\nu _{1}}^{\mu _{2},\nu _{2}}(\xi )\bigg{\{}%
([r+1]_{q}+\nu _{1})\left( \frac{[r+1]_{q}+\nu _{2}}{[r+1]_{q}}\right)
_{q}\int_{\frac{q[s]_{q}+\mu _{1}}{[r+1]_{q}+\nu _{1}}}^{\frac{[s+1]_{q}+\mu
_{1}}{[r+1]_{q}+\nu _{1}}}|f(t)|\mathrm{d}t\bigg{\}}^{p} \\
&\leq &\sum_{s=0}^{r}\left( \frac{[r+1]_{q}+\nu _{2}}{[r+1]_{q}}\right)
_{q}^{r}P_{r,q,\mu _{1},\nu _{1}}^{\mu _{2},\nu _{2}}(\xi )([r+1]_{q}+\nu
_{1})\left( \frac{[r+1]_{q}+\nu _{2}}{[r+1]_{q}}\right) _{q}^{p}\int_{\frac{%
q[s]_{q}+\mu _{1}}{[r+1]_{q}+\nu _{1}}}^{\frac{[s+1]_{q}+\mu _{1}}{%
[r+1]_{q}+\nu _{1}}}|f(t)|^{p}\mathrm{d}t.
\end{eqnarray*}%
Taking integral over $\mathcal{J}_{r}=\left[ \frac{\mu _{2}}{[r+1]_{q}+\nu
_{2}},\frac{[r+1]_{q}+\mu _{2}}{[r+1]_{q}+\nu _{2}}\right] ,$ we obtain

\begin{eqnarray*}
\int_{\mathcal{J}_{r}}|\mathcal{B}_{r,q,\mu _{1},\nu _{1}}^{\mu _{2},\nu
_{2}}(f;\xi )|^{p} &\leq &\sum_{s=0}^{r}\left( \frac{[r+1]_{q}+\nu _{2}}{%
[r+1]_{q}}\right) _{q}^{r}\left( \frac{[r+1]_{q}}{[r+1]_{q}+\nu _{2}}\right)
_{q}^{r+1}\frac{1}{[r+1]_{q}} \\
&\times &([r+1]_{q}+\nu _{1})\left( \frac{[r+1]_{q}+\nu _{2}}{[r+1]_{q}}%
\right) _{q}^{p}\int_{\frac{q[s]_{q}+\mu _{1}}{[r+1]_{q}+\nu _{1}}}^{\frac{%
[s+1]_{q}+\mu _{1}}{[r+1]_{q}+\nu _{1}}}|f(t)|^{p}\mathrm{d}t \\
&=&\frac{[r+1]_{q}+\nu _{1}}{[r+1]_{q}}\left( \frac{[r+1]_{q}+\nu _{2}}{%
[r+1]_{q}}\right) _{q}^{p-1}\sum_{s=0}^{r}\int_{\frac{q[s]_{q}+\mu _{1}}{%
[r+1]_{q}+\nu _{1}}}^{\frac{[s+1]_{q}+\mu _{1}}{[r+1]_{q}+\nu _{1}}%
}|f(t)|^{p}\mathrm{d}t \\
&\leq &\left( \frac{[r+1]_{q}+\nu _{2}}{[r+1]_{q}}\right)
_{q}^{p}||f||_{L_{p}[0,1]}^{p}.
\end{eqnarray*}

On the other side, from \eqref{tt}, if we use the inequality $%
\int_{[0,1]\setminus \mathcal{J}_{r}}|f(\xi )|^{p}\mathrm{d}z\leq
||f||_{L_{p}[0,1]}^{p}$. Then we conclude that
\begin{equation}
\int_{0}^{1}|\mathcal{C}_{r,q,\mu _{1},\nu _{1}}^{\mu _{2},\nu _{2}}(f;\xi
)|^{p}\mathrm{d}z\leq \bigg{[}1+\left( \frac{[r+1]_{q}+\nu _{2}}{[r+1]_{q}}%
\right) _{q}^{p}\bigg{]}||f||_{L_{p}[0,1]}^{p}.  \label{tn}
\end{equation}%
Thus

\begin{equation*}
||\mathcal{C}_{r,q,\mu _{1},\nu _{1}}^{\mu _{2},\nu
_{2}}(f)||_{L_{p}[0,1]}\leq (2+\nu _{2})||g||_{L_{p}[0,1]}\leq
R||f||_{L_{p}[0,1]}.
\end{equation*}%
Therefore, for all $r\in \mathbb{N},$ if $||\mathcal{C}_{r,q,\mu _{1},\nu
_{1}}^{\mu _{2},\nu _{2}}||_{L_{P}[0,1]}\leq R,$ then \eqref{mnn-1} gives us

\begin{eqnarray*}
||\mathcal{C}_{r,q,\mu _{1},\nu _{1}}^{\mu _{2},\nu
_{2}}(f)-f||_{L_{p}[0,1]} &\leq &||\mathcal{C}_{r,q,\mu _{1},\nu _{1}}^{\mu
_{2},\nu _{2}}(\psi )-\psi ||_{L_{p}[0,1]} \\
&+&||\mathcal{C}_{r,q,\mu _{1},\nu _{1}}^{\mu _{2},\nu _{2}}||\;||f-\psi
||_{L_{p}[0,1]}+||f-\psi ||_{L_{p}[0,1]} \\
&\leq &\epsilon R+2\epsilon .
\end{eqnarray*}%
Similarly, we also have

\begin{eqnarray*}
||\mathcal{C}_{r,q,\mu _{1},\nu _{1}}^{\mu _{2},\nu
_{2}}(f)-f||_{L_{p}[0,1]} &=&\bigg{(}\int_{0}^{1}|\mathcal{C}_{r,q,\mu
_{1},\nu _{1}}^{\mu _{2},\nu _{2}}(f;\xi )|^{p}\mathrm{d}z\bigg{)}^{\frac{1}{%
p}} \\
&=&\bigg{(}\int_{\mathcal{J}_{r}}|\mathcal{B}_{r,q,\mu _{1},\nu _{1}}^{\mu
_{2},\nu _{2}}(f;\xi )-f(\xi )|^{p}\mathrm{d}z\bigg{)}^{\frac{1}{p}} \\
&=&||\mathcal{B}_{r,q,\mu _{1},\nu _{1}}^{\mu _{2},\nu _{2}}(f;\xi )-f(\xi
)||_{L_{p}(\mathcal{J}_{r})} \\
&\leq &\epsilon R+2\epsilon .
\end{eqnarray*}%
Thus we get $\lim_{r\rightarrow \infty }||\mathcal{B}_{r,q,\mu _{1},\nu
_{1}}^{\mu _{2},\nu _{2}}(f;\xi )-f(\xi )||_{L_{p}(\mathcal{J}_{r})}=0$.
This completes the proof of Theorem \ref{sm-1}.
\end{proof}

\section{\textbf{Convergence properties of operators $\mathcal{B}_{r,q,%
\protect\mu _{1},\protect\nu _{1}}^{\protect\mu _{2},\protect\nu _{2}}$ }}

We write $\tilde{C}[0,1]$ for the space of uniformly continuous functions on
$[0,1]$ and $E_{f}=\{f~~~\big{ |}~~~f\in \tilde{C}[0,1]\}$. For $\tilde{%
\delta}>0$, let $\tilde{\omega}(f;\tilde{\delta})$ be the modulus of
continuity of the function $f$ of order one and $\lim_{\tilde{\delta}%
\rightarrow 0+}\tilde{\omega}(f;\tilde{\delta})=0$ . Thus we immediately see%
\newline
\begin{equation}
\tilde{\omega}(f;\tilde{\delta})=\sup_{\mid t_{1}-t_{2}\mid \leq \tilde{%
\delta}}\mid f(t_{1})-f(t_{2})\mid ;~~~t_{1},t_{2}\in \lbrack 0,1],
\label{snson1}
\end{equation}%
\begin{equation}
\mid f(t_{1})-f(t_{2})\mid \leq \left( 1+\frac{\mid t_{1}-t_{2}\mid }{\tilde{%
\delta}}\right) \tilde{\omega}(f;\tilde{\delta}).  \label{snson2}
\end{equation}

\begin{theorem}
\cite{snh} \label{ttm-1} Let $[u,v]\subseteq [x,y]$ and $\{K\}_{r \geq 1}$
be the sequence of positive linear operators acting from $C[x,y]$ to $C[u,v]$%
. Then

\begin{enumerate}
\item if $f\in C[x,y]$ and $\xi \in \lbrack u,v],$ then we have
\begin{eqnarray*}
|K_{r}(f;\xi )-f(\xi )| &\leq &|f(\xi )||K_{r}(\eta _{0}(t);\xi )-1| \\
&+&\big{\{}K_{r}(\eta _{0}(t);\xi )+\frac{1}{\tilde{\delta}}\sqrt{%
K_{r}((\eta _{1}(t)-\xi )^{2};\xi )}\sqrt{K_{r}(\eta _{0}(t);\xi )}\big{\}}%
\tilde{\omega}(f;\tilde{\delta}),
\end{eqnarray*}

\item for any $f^{\prime }\in C[x,y]$ and $\xi \in \lbrack u,v],$ we have
\begin{eqnarray*}
|K_{r}(f;\xi )-f(\xi )| &\leq &|f(\xi )||K_{r}(\eta _{0}(t);\xi
)-1|+|f^{\prime }(\xi )||K_{r}(\eta _{1}(t)-\xi ;\xi )| \\
&+&K_{r}((\eta _{1}(t)-\xi )^{2};\xi )\big{\{}\sqrt{K_{r}(\eta _{0}(t);\xi )}%
+\frac{1}{\tilde{\delta}}\sqrt{K_{r}((\eta _{1}(t)-\xi )^{2};\xi )}\big{\}}%
\tilde{\omega}(f^{\prime };\tilde{\delta}).
\end{eqnarray*}
\end{enumerate}
\end{theorem}

\begin{theorem}
\label{ttm-2} For all $f\in E_{f}$ and $\xi \in \mathcal{J}_{r}$ operators %
\eqref{d-1} satisfy the inequality
\begin{equation*}
|\mathcal{B}_{r,q,\mu _{1},\nu _{1}}^{\mu _{2},\nu _{2}}(f;\xi )-f(\xi
)|\leq \frac{\nu _{2}}{[r+1]_{q}}|f(\xi )|+2\left( \frac{[r+1]_{q}+\nu _{2}}{%
[r+1]_{q}}\right) _{q}\tilde{\omega}\left( f;\sqrt{\tilde{\delta}_{r,q,\mu
_{1},\nu _{1}}^{\mu _{2},\nu _{2}}(\xi )}\right) ,
\end{equation*}%
where $\tilde{\delta}_{r,q,\mu _{1},\nu _{1}}^{\mu _{2},\nu _{2}}(\xi
)=\left( \frac{[r+1]_{q}}{[r+1]_{q}+\nu _{2}}\right) _{q}\mathcal{B}%
_{r,q,\mu _{1},\nu _{1}}^{\mu _{2},\nu _{2}}((\eta _{1}(t)-\xi )^{2};\xi )$.
\end{theorem}

\begin{proof}
If we consider the (1) of theorem \ref{ttm-1} and lemma \ref{snlm1}, we can
write
\begin{eqnarray*}
|\mathcal{B}_{r,q,\mu _{1},\nu _{1}}^{\mu _{2},\nu _{2}}(f;\xi )-f(\xi )|
&\leq &|f(\xi )||\mathcal{B}_{r,q,\mu _{1},\nu _{1}}^{\mu _{2},\nu
_{2}}(\eta _{0}(t);\xi )-1|+\bigg{\{}\mathcal{B}_{r,q,\mu _{1},\nu
_{1}}^{\mu _{2},\nu _{2}}(\eta _{0}(t);\xi ) \\
&+&\frac{1}{\tilde{\delta}}\sqrt{\mathcal{B}_{r,q,\mu _{1},\nu _{1}}^{\mu
_{2},\nu _{2}}((\eta _{1}(t)-\xi )^{2};z)}\sqrt{\mathcal{B}_{r,q,\mu
_{1},\nu _{1}}^{\mu _{2},\nu _{2}}(\eta _{0}(t);\xi )}\bigg{\}}\tilde{\omega}%
(f;\tilde{\delta})
\end{eqnarray*}%
If we choose $\tilde{\delta}=\sqrt{\tilde{\delta}_{r,q,\mu _{1},\nu
_{1}}^{\mu _{2},\nu _{2}}(\xi )}=\sqrt{\left( \frac{[r+1]_{q}}{[r+1]_{q}+\nu
_{2}}\right) _{q}}\sqrt{\mathcal{B}_{r,q,\mu _{1},\nu _{1}}^{\mu _{2},\nu
_{2}}((\eta _{1}(t)-\xi )^{2};\xi )},$ then we get
\begin{equation*}
|\mathcal{B}_{r,q,\mu _{1},\nu _{1}}^{\mu _{2},\nu _{2}}(f;\xi )-f(\xi
)|\leq \frac{\nu _{2}}{[r+1]_{q}}|f(\xi )|+2\left( \frac{[r+1]_{q}+\nu _{2}}{%
[r+1]_{q}}\right) _{q}\tilde{\omega}\left( f;\sqrt{\tilde{\delta}_{r,q,\mu
_{1},\nu _{1}}^{\mu _{2},\nu _{2}}(\xi )}\right) .
\end{equation*}
\end{proof}

\begin{theorem}
\label{ttm-3} For all $f\in E_{f}$ and $\xi \in \mathcal{J}_{r}$ operators %
\eqref{d-1} verify the inequality
\begin{equation*}
|\mathcal{B}_{r,q,\mu _{1},\nu _{1}}^{\mu _{2},\nu _{2}}(f;\xi )-f(\xi
)|\leq \frac{\nu _{2}}{[r+1]_{q}}R_{r}+2\left( \frac{[r+1]_{q}+\nu _{2}}{%
[r+1]_{q}}\right) _{q}\tilde{\omega}\left( f;\sqrt{\tilde{\delta}}\right) ,
\end{equation*}%
where $R_{r}=\max_{\xi \in \mathcal{J}_{r}}|f(\xi )|$ and $\tilde{\delta}%
=\max_{\xi \in \mathcal{J}_{r}}\tilde{\delta}_{r,q,\mu _{1},\nu _{1}}^{\mu
_{2},\nu _{2}}(\xi )$.
\end{theorem}

\begin{proof}
In the view of monotonicity of the modulus of continuity, we easily led to
get desired result.
\end{proof}

\begin{remark}
\label{rmk-1} For all $\xi \in \mathcal{J}_{r}$ Theorem \ref{ttm-2}
estimates the local order approximation, while Theorem \ref{ttm-3} allows to
estimate the global order approximation.
\end{remark}

\begin{theorem}
\label{ttm-4} If $\varphi \in C^{\prime }[0,1]$ then for every $\xi \in
\mathcal{J}_{r},$ we have\newline
\newline
$|\mathcal{B}_{r,q,\mu _{1},\nu _{1}}^{\mu _{2},\nu _{2}}(\varphi ;\xi
)-\varphi (\xi )|$
\begin{eqnarray*}
&\leq &\frac{\nu _{2}}{[r+1]_{q}}|\varphi (\xi )|+\Bigg{|}\bigg{[}\frac{2q}{%
[2]_{q}}\frac{[r]_{q}}{[r+1]_{q}+\nu _{1}}\left( \frac{[r+1]_{q}+\nu _{2}}{%
[r+1]_{q}}\right) _{q}^{2}-\left( \frac{[r+1]_{q}+\nu _{2}}{[r+1]_{q}}%
\right) _{q}\bigg{]}\xi \\
&+&\bigg{[}\frac{1+2\mu _{1}}{[2]_{q}([r+1]_{q}+\nu _{1})}\left( \frac{%
[r+1]_{q}+\nu _{2}}{[r+1]_{q}}\right) _{q} \\
&-&\frac{2q\mu _{2}}{[2]_{q}}\frac{[r]_{q}}{([r+1]_{q}+\nu
_{1})([r+1]_{q}+\nu _{2})}\left( \frac{[r+1]_{q}+\nu _{2}}{[r+1]_{q}}\right)
_{q}^{2}\bigg{]}\Bigg{|}|\varphi ^{\prime }(\xi )| \\
&+&2\left( \frac{[r+1]_{q}+\nu _{2}}{[r+1]_{q}}\right) _{q}\sqrt{\tilde{%
\delta}_{r,q,\mu _{1},\nu _{1}}^{\mu _{2},\nu _{2}}}\tilde{\omega}\left(
\varphi ^{\prime };\sqrt{\tilde{\delta}_{r,q,\mu _{1},\nu _{1}}^{\mu
_{2},\nu _{2}}}(\xi )\right) ,
\end{eqnarray*}%
where $\tilde{\delta}_{r,q,\mu _{1},\nu _{1}}^{\mu _{2},\nu _{2}}$ is
defined in Theorem \ref{ttm-2}.
\end{theorem}

\begin{proof}
Takin into consideration (2) of Theorem \ref{ttm-1} and Lemma \ref{snlm1},
we get\newline
\newline
$|\mathcal{B}_{r,q,\mu _{1},\nu _{1}}^{\mu _{2},\nu _{2}}(\varphi ;\xi
)-\varphi (\xi )|$
\begin{eqnarray*}
&\leq &\frac{\nu _{2}}{[r+1]_{q}}|\varphi (\xi )|+\Bigg{|}\bigg{[}\frac{2q}{%
[2]_{q}}\frac{[r]_{q}}{[r+1]_{q}+\nu _{1}}\left( \frac{[r+1]_{q}+\nu _{2}}{%
[r+1]_{q}}\right) _{q}^{2}-\left( \frac{[r+1]_{q}+\nu _{2}}{[r+1]_{q}}%
\right) _{q}\bigg{]}\xi \\
&+&\bigg{[}\frac{1+2\mu _{1}}{[2]_{q}([r+1]_{q}+\nu _{1})}\left( \frac{%
[r+1]_{q}+\nu _{2}}{[r+1]_{q}}\right) _{q} \\
&-&\frac{2q\mu _{2}}{[2]_{q}}\frac{[r]_{q}}{([r+1]_{q}+\nu
_{1})([r+1]_{q}+\nu _{2})}\left( \frac{[r+1]_{q}+\nu _{2}}{[r+1]_{q}}\right)
_{q}^{2}\bigg{]}\Bigg{|}|\varphi ^{\prime }(\xi )| \\
&+&\sqrt{\mathcal{B}_{r,q,\mu _{1},\nu _{1}}^{\mu _{2},\nu _{2}}((\eta
_{1}(t)-\xi )^{2};\xi )}\sqrt{\mathcal{B}_{r,q,\mu _{1},\nu _{1}}^{\mu
_{2},\nu _{2}}((\eta _{0}(z);\xi )}\bigg{\{}1+\frac{1}{\tilde{\delta}}\frac{%
\sqrt{\mathcal{B}_{r,q,\mu _{1},\nu _{1}}^{\mu _{2},\nu _{2}}((\eta
_{1}(t)-\xi )^{2};\xi )}}{\sqrt{\mathcal{B}_{r,q,\mu _{1},\nu _{1}}^{\mu
_{2},\nu _{2}}((\eta _{0}(\xi );\xi )}}\bigg{\}}\tilde{\omega}\left( \varphi
^{\prime };\tilde{\delta}\right) .
\end{eqnarray*}%
In the view of theorem \ref{ttm-2} if we choose $\tilde{\delta}=\sqrt{\tilde{%
\delta}_{r,q,\mu _{1},\nu _{1}}^{\mu _{2},\nu _{2}}(\xi )}=\sqrt{\left(
\frac{[r+1]_{q}}{[r+1]_{q}+\nu _{2}}\right) _{q}}\sqrt{\mathcal{B}_{r,q,\mu
_{1},\nu _{1}}^{\mu _{2},\nu _{2}}((\eta _{1}(t)-\xi )^{2};\xi )},$ then we
get our results easily.
\end{proof}

By the virtue of some earlier information we want to emphasize in the
development to study approximation by positive linear operators \eqref{d-1}
in $L_{p}$ spaces.

In our main instruments we use space $L_{p}(\mathcal{J}_{r})$ of integral
modification in terms of modulus of continuity for all $\psi \in L_{p}(\Phi
_{\lambda })$ by
\begin{equation}
\tilde{\omega}_{1,p}(\psi ,t)=\sup_{z\in \lbrack 0,1]}\sup_{0<\lambda \leq
t}\parallel \psi (z+\lambda )-\psi (z)\parallel _{L_{p}(\Phi _{\lambda
})},\quad (1\leq p<\infty ),
\end{equation}%
where $\parallel .\parallel _{L_{p}(\Phi _{\lambda })}$ is the $L_{p}$-norm
defined over $\Phi _{\lambda }=[0,1-\lambda ]$. In addition to measure the
quantitative estimates by the Peetre's $K$-functional we let $\psi $ be
absolutely continuous function and $\mathcal{F}_{1,p}(\Phi _{\lambda
})=\{\psi ,\psi ^{\prime }\in L_{p}(\Phi _{\lambda })\}$. For any $\varphi
\in L_{p}(\Phi _{\lambda })$ and $1\leq p<\infty ,$ the Peetre's $K$%
-functional is given by

\begin{equation}
K_{1,p}(\psi ;t)=\inf_{\varphi \in \mathcal{F}_{1,p}(\Phi _{\lambda })}%
\bigg{(}\parallel \psi -\varphi \parallel _{L_{p}(\Phi _{\lambda
})}+t\parallel \varphi ^{\prime }\parallel _{L_{p}(\Phi _{\lambda })}\bigg{)}%
.
\end{equation}%
Next, the connection between the Peetre's $K$-functional and integral
modulus of continuity is given by the inequality \cite{jon} as follows
\begin{equation}
M_{1}\tilde{\omega}_{1,p}(\psi ;t)\leq K_{1,p}(\psi ;t)\leq M_{2}\tilde{%
\omega}_{1,p}(\psi ;t).
\end{equation}%
We denote $\mathcal{J}_{s}=\bigg{[}\frac{[s]_{q}+\mu _{1}}{[r+1]_{q}+\nu _{1}%
},\frac{[s+1]_{q}+\mu _{1}}{[r+1]_{q}+\nu _{1}}\bigg{]}$ and taking into
account the operators \eqref{d-1}, let us consider the following auxiliary
operators
\begin{equation}
\mathcal{D}_{r,q,\mu _{1},\nu _{1}}^{\mu _{2},\nu _{2}}(f;\xi )=\left( \frac{%
[r+1]_{q}}{[r+1]_{q}+\nu _{2}}\right) _{q}\mathcal{B}_{r,q,\mu _{1},\nu
_{1}}^{\mu _{2},\nu _{2}}(f;\xi ).  \label{st-0}
\end{equation}

\begin{theorem}
\label{st-1} For all $\varphi \in \mathcal{F}_{1,p}[0,1]$ and $p>1,$
operators $\mathcal{D}_{r,q,\mu _{1},\nu _{1}}^{\mu _{2},\nu _{2}}$ %
\eqref{st-0} satisfy the inequality

\begin{equation*}
\left\vert \left\vert \mathcal{D}_{r,q,\mu _{1},\nu _{1}}^{\mu _{2},\nu
_{2}}(\varphi ;\xi )-\varphi (\xi )\right\vert \right\vert _{L_{p}(\mathcal{J%
}_{s})}\leq 2^{\frac{1}{p}}\left( 1+\frac{1}{p-1}\right) \max_{\xi \in
\mathcal{J}_{s}}\bigg{(}\mathcal{D}_{r,q,\mu _{1},\nu _{1}}^{\mu _{2},\nu
_{2}}\left( (\eta _{1}(t)-\xi )^{2};\xi \right) \bigg{)}^{\frac{1}{2}%
}\left\vert \left\vert \varphi ^{\prime }\right\vert \right\vert
_{L_{p}[0,1]}
\end{equation*}%
where $\mathcal{D}_{r,q,\mu _{1},\nu _{1}}^{\mu _{2},\nu _{2}}\left( \eta
_{1}(t)-\xi )^{2};\xi \right) $ is defined by \eqref{st-0}.
\end{theorem}

\begin{proof}
We consider
\begin{equation}
\mathcal{Q}_{r,q,\mu _{1},\nu _{1}}^{\mu _{2},\nu _{2}}(\xi )=\left( \frac{%
[r+1]_{q}+\nu _{2}}{[r+1]_{q}}\right) _{q}^{r}\left[
\begin{array}{c}
r \\
s%
\end{array}%
\right] _{q}\left( \xi -\frac{\mu _{2}}{[r+1]_{q}+\nu _{2}}\right)
_{q}^{s}\left( \frac{[r+1]_{q}+\mu _{2}}{[r+1]_{q}+\nu _{2}}-\xi \right)
_{q}^{r-s}  \label{etn-1}
\end{equation}

For any $\xi \in \mathcal{J}_{s},$ we can write
\begin{eqnarray*}
\left\vert \mathcal{D}_{r,q,\mu _{1},\nu _{1}}^{\mu _{2},\nu _{2}}(\varphi
;\xi )-\varphi (\xi )\right\vert &=&([r+1]_{q}+\nu _{1})\left\vert
\sum_{s=0}^{r}\mathcal{Q}_{r,q,\mu _{1},\nu _{1}}^{\mu _{2},\nu _{2}}(\xi
)\int_{\mathcal{J}_{s}}\left( \varphi (t)-\varphi (\xi )\right) \mathrm{d}%
t\right\vert \\
&\leq &([r+1]_{q}+\nu _{1})\sum_{s=0}^{r}\mathcal{Q}_{r,q,\mu _{1},\nu
_{1}}^{\mu _{2},\nu _{2}}(\xi )\int_{\mathcal{J}_{s}}\int_{\xi
}^{t}\left\vert \varphi ^{\prime }(\zeta )\right\vert \mathrm{d}\zeta
\mathrm{d}t \\
&\leq &\Theta _{\varphi ^{\prime }}(\xi )([r+1]_{q}+\nu _{1})\sum_{s=0}^{r}%
\mathcal{Q}_{r,q,\mu _{1},\nu _{1}}^{\mu _{2},\nu _{2}}(\xi )\int_{\mathcal{J%
}_{s}}\left\vert t-\xi \right\vert \mathrm{d}t,
\end{eqnarray*}%
where $\Theta _{\varphi ^{\prime }}(\xi )=\sup_{t\in \lbrack 0,1]}\frac{1}{%
t-\xi }\int_{z}^{t}\left\vert \varphi ^{\prime }(\zeta )\right\vert \mathrm{d%
}\lambda $ $(t\neq \xi )$ denotes the Hardy-Littlewood majorant of $\varphi
^{\prime }$. By the virtue of the well-known Cauchy-Schwarz's inequality, we
immediately get
\begin{eqnarray*}
\left\vert \mathcal{D}_{r,q,\mu _{1},\nu _{1}}^{\mu _{2},\nu _{2}}(\varphi
;\xi )-\varphi (\xi )\right\vert &\leq &\Theta _{\varphi ^{\prime }}(\xi
)([r+1]_{q}+\nu _{1})^{\frac{1}{2}}\left( \sum_{s=0}^{r}\mathcal{Q}_{r,q,\mu
_{1},\nu _{1}}^{\mu _{2},\nu _{2}}(\xi )\right) ^{\frac{1}{2}} \\
&\times &\bigg{(}\sum_{s=0}^{r}\mathcal{Q}_{r,q,\mu _{1},\nu _{1}}^{\mu
_{2},\nu _{2}}(\xi )\int_{\mathcal{J}_{s}}\left( \eta _{1}(t)-\xi \right)
^{2}\mathrm{d}t\bigg{)}^{\frac{1}{2}} \\
&\leq &\Theta _{\varphi ^{\prime }}(z)\max_{z\in \mathcal{J}_{s}}\bigg{(}%
\mathcal{D}_{r,q,\mu _{1},\nu _{1}}^{\mu _{2},\nu _{2}}\left( (\eta
_{1}(t)-\xi )^{2};\xi \right) \bigg{)}^{\frac{1}{2}},
\end{eqnarray*}

For $1<p<\infty,$ the Hardy-Littlewood theorem \cite{ygm} gives the
inequality

\begin{equation*}
\int_{0}^{1}\Theta _{\varphi ^{\prime }}(\xi )\mathrm{d}\xi \leq 2\left(
\frac{p}{p-1}\right) ^{p}\int_{0}^{1}\left\vert \varphi ^{\prime }(\xi
)\right\vert ^{p}\mathrm{d}z.
\end{equation*}%
Thus we get

\begin{equation*}
\left\vert \left\vert \mathcal{D}_{r,q,\mu _{1},\nu _{1}}^{\mu _{2},\nu
_{2}}(\varphi )-\varphi \right\vert \right\vert _{L_{p}(\mathcal{J}%
_{s})}\leq 2^{\frac{1}{p}}\left( 1+\frac{1}{p-1}\right) \max_{z\in \mathcal{J%
}_{s}}\bigg{(}\mathcal{D}_{r,q,\mu _{1},\nu _{1}}^{\mu _{2},\nu _{2}}\left(
(\eta _{1}(t)-\xi )^{2};\xi \right) \bigg{)}^{\frac{1}{2}}\left\vert
\left\vert \varphi ^{\prime }\right\vert \right\vert _{L_{p}[0,1]}
\end{equation*}
\end{proof}

\begin{theorem}
\label{stt-1} For all $\varphi \in L_{p}[0,1],$ operators \eqref{st-0}
satisfy
\begin{equation*}
\left\vert \left\vert \mathcal{D}_{r,q,\mu _{1},\nu _{1}}^{\mu _{2},\nu
_{2}}(\varphi )-\varphi \right\vert \right\vert _{L_{p}(\mathcal{J}%
_{s})}\leq 2M_{2}\bigg{(}1+2^{\frac{1-p}{p}}\left( 1+\frac{1}{p-1}\right)
\tilde{\omega}_{1,p}\left( \varphi ;\rho _{r,q,\mu _{1},\nu _{1}}^{\mu
_{2},\nu _{2}}(\xi )\right) \bigg{)},
\end{equation*}%
where $M_{2}$ is positive constant and $\rho \xi _{r,q,\mu _{1},\nu
_{1}}^{\mu _{2},\nu _{2}}(\xi )=\max_{\xi \in \mathcal{J}_{s}}\bigg{(}%
\mathcal{D}_{r,q,\mu _{1},\nu _{1}}^{\mu _{2},\nu _{2}}\left( (\eta
_{1}(t)-\xi )^{2};\xi \right) \bigg{)}^{\frac{1}{2}}$ by Theorem \ref{st-1}.
\end{theorem}

\begin{proof}
We consider
\begin{equation}
\left\vert \left\vert \mathcal{D}_{r,q,\mu _{1},\nu _{1}}^{\mu _{2},\nu
_{2}}(\varphi )-\varphi \right\vert \right\vert _{L_{p}[0,1]}\leq
\begin{cases}
2\left\vert \left\vert \varphi \right\vert \right\vert _{L_{p}[0,1]} & \quad
\text{if }\;\varphi \in L_{p}[0,1], \\
2^{\frac{1}{p}}\left( \frac{p}{p-1}\right) \rho _{r,q,\mu _{1},\nu
_{1}}^{\mu _{2},\nu _{2}}(\xi )\left\vert \left\vert \varphi \right\vert
\right\vert _{L_{p}[0,1]} & \quad \text{if }\;\varphi \in \mathcal{F}%
_{1,p}[0,1],%
\end{cases}
\label{tot}
\end{equation}%
where as in Theorem \ref{st-1} we suppose $\rho _{r,q,\mu _{1},\nu
_{1}}^{\mu _{2},\nu _{2}}(\xi )=\max_{\xi \in \mathcal{J}_{s}}\bigg{(}%
\mathcal{D}_{r,q,\mu _{1},\nu _{1}}^{\mu _{2},\nu _{2}}\left( (\eta
_{1}(t)-\xi )^{2};\xi \right) \bigg{)}^{\frac{1}{2}}$.\newline

Thus for an arbitrary function $\psi\in \mathcal{F}_{1,p}[0,1]$ operators %
\label{tot} verify that

\begin{eqnarray*}
&&\left\vert \left\vert \mathcal{D}_{r,q,\mu _{1},\nu _{1}}^{\mu _{2},\nu
_{2}}(\varphi )-\varphi \right\vert \right\vert _{L_{p}(\mathcal{J}_{s})} \\
&\leq &2\bigg{\{}\left\vert \left\vert \varphi -\psi \right\vert \right\vert
_{L_{p}[0,1]}+2^{\frac{1-p}{p}}\left( 1+\frac{1}{p-1}\right) \rho _{r,q,\mu
_{1},\nu _{1}}^{\mu _{2},\nu _{2}}(\xi )\left\vert \left\vert \psi ^{\prime
}\right\vert \right\vert _{L_{p}[0,1]}\bigg{\}} \\
&\leq &2K_{1,p}\bigg{\{}\varphi ;2^{\frac{1-p}{p}}\left( 1+\frac{1}{p-1}%
\right) \rho _{r,q,\mu _{1},\nu _{1}}^{\mu _{2},\nu _{2}}(\xi )\bigg{\}} \\
&\leq &2M_{2}\tilde{\omega}_{1,p}\bigg{\{}\varphi ;2^{\frac{1-p}{p}}\left( 1+%
\frac{1}{p-1}\right) \rho _{r,q,\mu _{1},\nu _{1}}^{\mu _{2},\nu _{2}}(\xi )%
\bigg{\}} \\
&\leq &2M_{2}\bigg{\{}1+2^{\frac{1-p}{p}}\left( 1+\frac{1}{p-1}\right) %
\bigg{\}}\tilde{\omega}_{1,p}\left( \varphi ;\rho _{r,q,\mu _{1},\nu
_{1}}^{\mu _{2},\nu _{2}}(\xi )\right) .
\end{eqnarray*}
\end{proof}

Now we give the local direct estimate for the operators $\mathcal{D}%
_{r,q,\mu _{1},\nu _{1}}^{\mu _{2},\nu _{2}}$ defined by \eqref{st-0} via
the well-known Lipschitz-type maximal function involving the parameters $%
\alpha ,\;\beta >0$ and number $\sigma \in (0,1]$. Thus from \cite{ozaraktu}
we recall that
\begin{equation*}
Lip_{K}^{(\alpha ,\beta )}(\sigma ):=\Big\{f\in C[0,1]:|f(t)-f(\xi )|\leq K%
\frac{|t-\xi |^{\sigma }}{(\alpha \xi ^{2}+\beta \xi +t)^{\frac{\sigma }{2}}}%
;~z,\;t\in \lbrack 0,1]\Big\},
\end{equation*}%
where $K$ is a positive constant .

\begin{theorem}
\label{local} For any $f\in Lip_{K}^{(\alpha ,\beta )}(\sigma )$ and $\sigma
\in (0,1]$, there exits positive $K$ such that
\begin{equation*}
|\mathcal{D}_{r,q,\mu _{1},\nu _{1}}^{\mu _{2},\nu _{2}}\left( f;\xi \right)
-f(\xi )|\leq K.(\alpha \xi ^{2}+\beta \xi )^{-\sigma /2}\left[ \mathcal{D}%
_{r,q,\mu _{1},\nu _{1}}^{\mu _{2},\nu _{2}}((\eta _{1}(t)-\xi )^{2};\xi )%
\right] ^{\frac{\sigma }{2}},
\end{equation*}%
where $\mathcal{D}_{r,q,\mu _{1},\nu _{1}}^{\mu _{2},\nu _{2}}\left( \eta
_{1}(t)-\xi )^{2};\xi \right) $ defined by \eqref{st-0}.
\end{theorem}

\begin{proof}
For any $f\in Lip_{K}^{(\alpha ,\beta )}(\sigma )$ and $\sigma \in (0,1],$
first we check the statement holds for $\sigma =1$. Then, in conclusion %
\eqref{etn-1} we can see that
\begin{align*}
|\mathcal{D}_{r,q,\mu _{1},\nu _{1}}^{\mu _{2},\nu _{2}}\left( f;\xi \right)
-f(\xi )|& \leq |\mathcal{D}_{r,q,\mu _{1},\nu _{1}}^{\mu _{2},\nu
_{2}}(|f(t)-f(\xi )|;\xi )|+f(\xi )~|\mathcal{D}_{r,q,\mu _{1},\nu
_{1}}^{\mu _{2},\nu _{2}}(\eta _{0}(t);\xi )-1| \\
& \leq ([r+1]_{q}+\nu _{1})\sum_{s=0}^{r}\bigg|f\left( t\right) -f(\xi )%
\bigg|~\mathcal{Q}_{r,q,\mu _{1},\nu _{1}}^{\mu _{2},\nu _{2}}(\xi ) \\
& \leq K([r+1]_{q}+\nu _{1})\sum_{s=0}^{r}\frac{|t-z|}{(\alpha z^{2}+\beta
z+t)^{\frac{1}{2}}}~\mathcal{Q}_{r,q,\mu _{1},\nu _{1}}^{\mu _{2},\nu
_{2}}(\xi ).
\end{align*}%
For any $\alpha ,\;\beta \geq 0,$ we use the inequality $(\alpha z^{2}+\beta
z+t)^{-1/2}\leq (\alpha z^{2}+\beta z)^{-1/2}$ and apply the well-known
Cauchy-Schwarz inequality, we get
\begin{align*}
|\mathcal{D}_{r,q,\mu _{1},\nu _{1}}^{\mu _{2},\nu _{2}}\left( f;\xi \right)
-f(\xi )|& \leq K([r+1]_{q}+\nu _{1})(\alpha z^{2}+\beta
z)^{-1/2}\sum_{s=0}^{r}|t-z|~\mathcal{Q}_{r,q,\mu _{1},\nu _{1}}^{\mu
_{2},\nu _{2}}(\xi ) \\
& =K.(\alpha z^{2}+\beta z)^{-1/2}|\mathcal{D}_{r,q,\mu _{1},\nu _{1}}^{\mu
_{2},\nu _{2}}(\eta _{1}(t)-\xi ;\xi )| \\
& \leq K\big{|}\mathcal{D}_{r,q,\mu _{1},\nu _{1}}^{\mu _{2},\nu _{2}}((\eta
_{1}(t)-\xi )^{2};\xi )\big{|}^{1/2}(\alpha \xi ^{2}+\beta \xi )^{-1/2}.
\end{align*}%
These conclusions imply that it is true for $\sigma =1$. Now we want to show
the statement is valid for $\eta \in (0,1)$. We apply the monotonicity
property to operators $\mathcal{D}_{r,q,\mu _{1},\nu _{1}}^{\mu _{2},\nu
_{2}}$ and use the H\"{o}lder's inequality two times with $c=2/\sigma $ and $%
d=2/(2-\sigma )$. Thus, we get
\begin{eqnarray*}
\left\vert \mathcal{D}_{r,q,\mu _{1},\nu _{1}}^{\mu _{2},\nu _{2}}\left(
f;\xi \right) -f(\xi )\right\vert &\leq &([r+1]_{q}+\nu _{1})\sum_{s=0}^{r}%
\bigg|f\left( t\right) -f(\xi )\bigg|~\mathcal{Q}_{r,q,\mu _{1},\nu
_{1}}^{\mu _{2},\nu _{2}}(\xi ) \\
&\leq &\bigg(\sum_{s=0}^{r}\bigg|f\left( t\right) -f(\xi )\bigg|^{\frac{2}{%
\sigma }}([r+1]_{q}+\nu _{1})\mathcal{Q}_{r,q,\mu _{1},\nu _{1}}^{\mu
_{2},\nu _{2}}(\xi )\bigg)^{\frac{\sigma }{2}} \\
&\times &\bigg(\sum_{s=0}^{r}([r+1]_{q}+\nu _{1})\mathcal{Q}_{r,q,\mu
_{1},\nu _{1}}^{\mu _{2},\nu _{2}}(\xi )\bigg)^{\frac{2-\sigma }{2}} \\
&\leq &K\bigg(\sum_{s=0}^{r}\frac{\big(t-z\big)^{2}([r+1]_{q}+\nu _{1})%
\mathcal{Q}_{r,q,\mu _{1},\nu _{1}}^{\mu _{2},\nu _{2}}(\xi )}{t+\alpha \xi
^{2}+\beta \xi }\bigg)^{\frac{\sigma }{2}} \\
&\leq &K(\alpha \xi ^{2}+\beta \xi )^{-\sigma /2}\bigg\{\sum_{s=0}^{r}\big(%
t-\xi \big)^{2}~([r+1]_{q}+\nu _{1})\mathcal{Q}_{r,q,\mu _{1},\nu _{1}}^{\mu
_{2},\nu _{2}}(\xi )\bigg\}^{\frac{\sigma }{2}} \\
&\leq &K(\alpha \xi ^{2}+\beta \xi )^{-\sigma /2}\left[ \mathcal{D}_{r,q,\mu
_{1},\nu _{1}}^{\mu _{2},\nu _{2}}((\eta _{1}(t)-\xi )^{2};\xi )\right] ^{%
\frac{\sigma }{2}}.
\end{eqnarray*}%
This completes the proof.
\end{proof}

Here, we establish a quantitative Voronovskaja-type theorem for the
operators $\mathcal{D}_{r,q,\mu _{1},\nu _{1}}^{\mu _{2},\nu _{2}}\left(
f;\xi \right) $. For any $f\in C[0,1]$, the modulus of smoothness is defined
by
\begin{equation}
\omega _{\beta }(f,\tilde{\delta})=\sup_{0<|\rho |\leq \tilde{\delta}}\bigg\{%
\bigg|f\bigg(\xi +\frac{\rho \beta (\xi )}{2}\bigg)-f\bigg(\xi -\frac{\rho
\beta (\xi )}{2}\bigg)\bigg|,\xi \pm \frac{\rho \beta (\xi )}{2}\in \lbrack
0,1]\bigg\},  \label{smooth}
\end{equation}%
where $\beta (\xi )=(\xi -\xi ^{2})^{\frac{1}{2}}$ and related Peetre's $K$%
-functional is given as
\begin{equation*}
\mathcal{K}_{\beta }(f,\tilde{\delta})=\inf_{g\in \mathcal{W}_{\beta }[0,1]}%
\big\{||f-g||+\tilde{\delta}||\beta g^{\prime }||:g\in C^{1}[0,1],\tilde{%
\delta}>0\big\},
\end{equation*}%
and $\mathcal{W}_{\beta }[0,1]=\{g:~g\in C_{A}[0,1],~\Vert \beta g^{\prime
}\Vert <\infty \}$. Let $C_{A}[0,1]$ be the class of absolutely continuous
functions defined on $[0,1]$. There is a positive constant $\mathcal{C}$
such that
\begin{equation*}
\mathcal{K}_{\beta }(f,\tilde{\delta})\leq \mathcal{C}~\omega _{\beta }(f,%
\tilde{\delta}).
\end{equation*}

\begin{theorem}
\label{thm10} Let $f,f^{\prime },f^{\prime \prime }\in C[0,1],$ then for
every $\xi \in \lbrack 0,1]$ it follows that
\begin{eqnarray*}
&&\bigg|\mathcal{D}_{r,q,\mu _{1},\nu _{1}}^{\mu _{2},\nu _{2}}\left( f;\xi
\right) -f(\xi )-\mathcal{D}_{r,q,\mu _{1},\nu _{1}}^{\mu _{2},\nu
_{2}}(\eta _{1}(t)-\xi ;\xi )f^{\prime }(\xi )-\frac{f^{\prime \prime }(\xi )%
}{2}\big(\mathcal{D}_{r,q,\mu _{1},\nu _{1}}^{\mu _{2},\nu _{2}}((\eta
_{1}(t)-\xi )^{2};\xi )+1\bigg| \\
&\leq &\mathcal{C}.O\left( \frac{1}{[r]_{q}}\right) \lambda ^{2}(\xi )\omega
_{\lambda }\bigg(f^{\prime \prime },[r]_{q}^{-\frac{1}{2}}\bigg),
\end{eqnarray*}%
where $\mathcal{D}_{r,q,\mu _{1},\nu _{1}}^{\mu _{2},\nu _{2}}$ is defined
by \eqref{st-0} and $\lambda (\xi )=\xi +1$.
\end{theorem}

\begin{proof}
For any $f\in C[0,1]$ we consider
\begin{equation*}
f(t)-f(\xi )-(t-\xi )f^{\prime }(\xi )=\int_{z}^{t}(t-\gamma )f^{\prime
\prime }(\gamma )\mathrm{d}\gamma
\end{equation*}%
Therefore, we can write
\begin{equation*}
f(t)-f(\xi )-(t-\xi )f^{\prime }(\xi )-\frac{f^{\prime \prime }(\xi )}{2}%
\big((t-\xi )^{2}+1\big)\leq \int_{z}^{t}(t-\gamma )[f^{\prime \prime
}(\gamma )-f^{\prime \prime }(\xi )]\mathrm{d}\gamma .
\end{equation*}%
On applying the operators $\mathcal{D}_{r,q,\mu _{1},\nu _{1}}^{\mu _{2},\nu
_{2}}\left( f;\xi \right) $, we obtain
\begin{align}
& \bigg|\mathcal{D}_{r,q,\mu _{1},\nu _{1}}^{\mu _{2},\nu _{2}}\left( f;\xi
\right) -f(\xi )-\mathcal{D}_{r,q,\mu _{1},\nu _{1}}^{\mu _{2},\nu
_{2}}(\eta _{1}(t)-\xi ;\xi )f^{\prime }(\xi )  \notag \\
& -\frac{f^{\prime \prime }(\xi )}{2}\big(\mathcal{D}_{r,q,\mu _{1},\nu
_{1}}^{\mu _{2},\nu _{2}}((\eta _{1}(t)-\xi )^{2};\xi )+\mathcal{D}_{r,q,\mu
_{1},\nu _{1}}^{\mu _{2},\nu _{2}}(\eta _{0}(t);\xi )\big)\bigg| \\
& \leq \mathcal{D}_{r,q,\mu _{1},\nu _{1}}^{\mu _{2},\nu _{2}}\bigg(\Big|%
\int_{z}^{t}|t-\gamma |~|f^{\prime \prime }(\gamma )-f^{\prime \prime }(\xi
)|~\mathrm{d}\gamma \Big|;z\bigg).
\end{align}%
We can estimate the right hand side expression such as
\begin{equation*}
\bigg|\int_{z}^{t}|t-\gamma |~|f^{\prime \prime }(\gamma )-f^{\prime \prime
}(\xi )|~\mathrm{d}\gamma \bigg|\leq 2\Vert f^{\prime \prime }-g\Vert (t-\xi
)^{2}+2\Vert \lambda g^{\prime }\Vert \lambda ^{-1}(\xi )|t-\xi |^{3}.
\end{equation*}%
We easily conclude that
\begin{equation*}
\mathcal{D}_{r,q,\mu _{1},\nu _{1}}^{\mu _{2},\nu _{2}}((\eta _{1}(t)-\xi
)^{2};\xi )\leq O\left( \frac{1}{[r]_{q}}\right) \lambda ^{2}(\xi )\;\text{%
and}\;\mathcal{D}_{r,q,\mu _{1},\nu _{1}}^{\mu _{2},\nu _{2}}((\eta
_{1}(t)-\xi )^{4};\xi )\leq O\left( \frac{1}{[r]_{q}^{2}}\right) \lambda
^{4}(\xi ),
\end{equation*}%
where $\lambda (\xi )=\xi +1$. From Cauchy-Schwarz inequality, we get
\begin{align*}
& \bigg|\mathcal{D}_{r,q,\mu _{1},\nu _{1}}^{\mu _{2},\nu _{2}}\left( f;\xi
\right) -f(\xi )-\mathcal{D}_{r,q,\mu _{1},\nu _{1}}^{\mu _{2},\nu
_{2}}(\eta _{1}(t)-\xi ;\xi )f^{\prime }(\xi ) \\
& -\frac{f^{\prime \prime }(\xi )}{2}\big(\mathcal{D}_{r,q,\mu _{1},\nu
_{1}}^{\mu _{2},\nu _{2}}((\eta _{1}(t)-\xi )^{2};\xi )+\mathcal{D}_{r,q,\mu
_{1},\nu _{1}}^{\mu _{2},\nu _{2}}(\eta _{0}(t);\xi )\big)\bigg| \\
& \leq 2\Vert f^{\prime \prime }-g\Vert \mathcal{D}_{r,q,\mu _{1},\nu
_{1}}^{\mu _{2},\nu _{2}}((\eta _{1}(t)-\xi )^{2};\xi )+2\Vert \lambda
g^{\prime }\Vert \lambda ^{-1}(\xi )\mathcal{D}_{r,q,\mu _{1},\nu _{1}}^{\mu
_{2},\nu _{2}}(|\eta _{1}(t)-\xi |^{3};\xi ) \\
& \leq 2.O\left( \frac{1}{[r]_{q}}\right) \gamma ^{2}(\xi )\Vert f^{\prime
\prime }-g\Vert \\
& +2\Vert \lambda g^{\prime }\Vert \lambda ^{-1}(\xi )\{\mathcal{D}_{r,q,\mu
_{1},\nu _{1}}^{\mu _{2},\nu _{2}}((\eta _{1}(t)-\xi )^{2};\xi )\}^{1/2}\{%
\mathcal{D}_{r,q,\mu _{1},\nu _{1}}^{\mu _{2},\nu _{2}}((\eta _{1}(t)-\xi
)^{4};\xi )\}^{1/2} \\
& \leq 2.O\left( \frac{1}{[r]_{q}}\right) \lambda ^{2}(\xi )\Big\{\Vert
f^{\prime \prime }-g\Vert +[r]_{q}^{-\frac{1}{2}}\Vert \lambda g^{\prime
}\Vert \Big\}.
\end{align*}%
By taking infimum over all $g\in W_{\lambda }[0,1]$, we easily deduce that
\begin{eqnarray*}
&&\bigg|\mathcal{D}_{r,q,\mu _{1},\nu _{1}}^{\mu _{2},\nu _{2}}\left( f;\xi
\right) -f(\xi )-\mathcal{D}_{r,q,\mu _{1},\nu _{1}}^{\mu _{2},\nu
_{2}}(\eta _{1}(t)-\xi ;\xi )f^{\prime }(\xi )-\frac{f^{\prime \prime }(\xi )%
}{2}\big(\mathcal{D}_{r,q,\mu _{1},\nu _{1}}^{\mu _{2},\nu _{2}}((\eta
_{1}(t)-\xi )^{2};\xi +1\bigg| \\
&\leq &\mathcal{C}.O\left( \frac{1}{[r]_{q}}\right) \lambda ^{2}(\xi )\omega
_{\gamma }\bigg(f^{\prime \prime },[r]_{q}^{-\frac{1}{2}}\bigg),
\end{eqnarray*}%
where $\omega _{\lambda }$ is the modulus of smoothness defined by %
\eqref{smooth}. Thus we complete the proof of Theorem \ref{thm10}.
\end{proof}

\section{Conclusion}

It is very clear that, the choice for $q=1$ in the operators \eqref{d-1},
coincide with the recent operators \cite{nfs} and the polynomials of our new
$q$-Bernstein-Stancu-Kantorovich variant of shifted knots operators become
the polynomials of Bernstein-Stancu-Kantorovich polynomials by \cite{nfs}.
Further, in case of $\mu _{2}=\nu _{2}=0$ with $q=1$ then, operators %
\eqref{d-1} reduced to \eqref{mnz-1} by \cite{barbosu}. Moreover, for $\mu
_{1}=\mu _{2}=\nu _{1}=\nu _{2}=0$ with $q=1$ our new operators \eqref{d-1}
coincide with the classic Bernstein-Kantorovich operators by \cite{kant}.
Thus in our investigation, we can say that the operators defined by \cite%
{barbosu,kant,nfs} are special cases of our $q$-Bernstein-Stancu-Kantorovich
variant of shifted knots operators \eqref{d-1}. 
\newline

\bigskip

\parindent=0mm\textbf{Acknowledgements}

\parindent=0mm Not applicable.\newline

\parindent=0mm\textbf{Availability of data and material}

\parindent=0mm Not applicable.\newline

\parindent=0mm\textbf{Competing interests}

\parindent=0mm The authors declare that they have no competing interests.%
\newline

\parindent=0mm\textbf{Funding}

\parindent=0mm Not applicable.\newline

\parindent=0mm\textbf{Authors' contributions}

\parindent=0mm All the authors contributed equally and significantly in
writing this paper.\newline

\end{document}